\documentclass[12pt]{amsart}
\usepackage[T2A]{fontenc}
\usepackage[english]{babel}
\usepackage{amsfonts,amssymb}

\usepackage[left=2cm,right=2cm, top=2cm,bottom=2cm,bindingoffset=0cm]{geometry}

\newtheorem{theorem}{Theorem}
\newtheorem{proposition}{Proposition}
\newtheorem{corollary}{Corollary}

\title{Morse index of saddle equilibria of gradient-like flows on connected sums of $\mathbb{S}^{n-1}\times \mathbb{S}^1$}
\author{Vyacheslav Grines}
\address{National Research University higher School of Economics, Nizhnii Novgorod}
\email{vgrines@yandex.ru}

\author{Elena Gurevich}
\address{National Research University higher School of Economics, Nizhnii Novgorod}
\email{egurevich@hse.ru}

\author{Sergiy Maksymenko}
\address{Institute of Mathematics NAS of Ukraine, Kyiv}
\email{maks@imath.kiev.ua}

\keywords{Homology of connected sum of $S^{n-1}\times S^1$, Morse index, gradient-like flows, topological classification}

\thanks{Research of V.~Grines and E.~Gurevich is done under financial support of RFS, grant 21-11-00010.}

\begin{document}

\begin{abstract}
Let $M$ be either $n$-sphere $\mathbb{S}^{n}$ or a connected sum of finitely many copies of $\mathbb{S}^{n-1}\times \mathbb{S}^{1}$, $n\geq4$.
A flow $f^t$ on $M$ is called gradient-like whenever its non-wandering set consists of finitely many hyperbolic equilibria and their invariant manifolds intersects transversally.
We prove that if invariant manifolds of distinct saddles of a gradient-like flow $f^t$ on $M$ do not intersect each other (in other words, $f^t$ has no heteroclinic intersections), then for each saddle of $f^t$ its Morse index (i.e.\ dimension of the unstable manifold) is either $1$ or $n-1$, so there are no saddles with Morse indices $i\in\{2,\ldots,n-2\}$.
\end{abstract}

\maketitle

\section{Introduction and statements of results}

Recall that a flow $f^t$ on a smooth closed manifold is called {\it gradient-like} if its non-wandering set consists of finitely many hyperbolic equilibria and their invariant manifolds have only transversal intersections.
In~\cite{Sm61} Smale proved that any closed manifold admits a Morse function (that is a smooth function all of whose critical points are non-generated) such that its gradient flow with respect to some Riemannian metric, is gradient-like.
In~\cite{Smale60a} Smale also established Morse inequalities for gradient-like flows.
Those inequalities relate the structure of the non-wandering set of a flow with the topology of the ambient manifold.
As Proposition~\ref{manman} below shows, this connection becomes more clear for gradient-like flows without heteroclinic intersection.

Let $M^n$ be a connected closed oriented manifold of dimension $n\geq 2$.
We will say that a gradient-like flow $f^t$ on $M^n$ belongs to the class $G(M^n)$ whenever the following two conditions hold:
\begin{itemize}
\item[(a)]
Morse index (i.e.\ dimension of the unstable manifold) of a saddle equilibrium state of the flow $f^t$ equals either $1$ or $n-1$;
\item[(b)]
invariant manifolds of distinct saddle equilibria do not intersect each other.
\end{itemize}

Denote by $\nu_{f^t}$ and  $\mu_{f^t}$ the numbers of saddle and node equilibria of the flow $f^t\in G(M^n)$ and set
\[
    g_{f^t}=(\nu_{f^t}-\mu_{f^t}+2)/2.
\]

Everywhere below $\mathcal{S}^n_g$ stands for a manifold being homeomorphic either to the sphere $\mathbb{S}^n$ if $g=0$ or to a connected sum of $g>0$ copies of $\mathbb{S}^{n-1}\times \mathbb{S}^{1}$.

\begin{proposition}\label{manman}
Let $f^t\in G(M^n)$, $n\geq 2$ and $g=g_{f^t}$.
Then $M^n$ is homeomorphic to $\mathcal{S}^n_{g}$.
\end{proposition}

For $n=2$ Proposition~\ref{manman} immediately follows from~\cite{Smale60a}, and for $n\geq 3$ from~\cite{BGMP02, GrGuPo-matan}, where its analogue for Morse-Smale diffeomorphisms without heteroclinic curves was obtained.

\begin{theorem}\label{good_index}
Let $f^t$ be a gradient-like flow on $\mathcal{S}^n_g$, $g\geq 0$, $n\geq4$.
If invariant manifolds of distinct saddle equilibria of $f^t$ do not intersect each other, then the Morse index of each saddle equilibrium equals either $1$ or $(n-1)$, that is $f^t\in G(\mathcal{S}^n_g)$.
Moreover, there exists $k\geq0$ such that $\nu_{f^t} = 2g+k$ and $\mu_{f^t}=k+2$.
\end{theorem}

Thus Theorem~\ref{good_index} claims that for manifolds $\mathcal{S}^n_g$, $n\geq4$, condition (b) implies condition (a).

For case $g=0$ Theorem~\ref{good_index} is proved in~\cite{Pil78}, where necessary and sufficient conditions of topological equivalence of flows from class $G(S^n)$, $n\geq 3$, are obtained.
Theorem~\ref{good_index} will allow to obtain topological classification of flows from class $\mathcal{S}^n_g$, $g>0$, in combinatorial terms using techniques of~\cite{Pil78, GrGuPoMa-non}.

\section{Homology of connected sums of $g$ copies of $\mathbb{S}^{n-1}\times \mathbb{S}^1$}
Everywhere below we consider homology only with integer coefficients.
Our aim is to prove that $H_i(\mathcal{S}^n_g)=0$ for $i\in\{2,\ldots,n-2\}$, $n\geq4$, see\ Corollary~\ref{invind} below.

It will be convenient to say that a topological space $X$ belongs to a class $\mathcal{Z}$ whenever its singular homology groups (with integer coefficients) are free abelian groups of finite ranks and distinct from zero only at finitely many dimensions.
Set $\beta_i(X) = \mathrm{rank}\,H_i(X)$ for $i\geq0$.
Then we have a well defined finite sum
\[
    p_X(t) = \sum_{i}\beta_i(X) t^i,
\]
called the \emph{Poincare polynomial} of the space $X$.

In particular, if $X,Y\in\mathcal{Z}$, then K\"{u}nneth theorem implies that
\[
    H_i(X\times Y)= \mathop{\oplus}\limits_{a+b=i} H_a(X) \otimes H_b (Y)
\]
for all $i\geq0$ (see, for example,~\cite[Corollary (29.11.1)]{GeenbergHarper}).
Since homology groups of $X$ and $Y$ are supposed to be free abelian, it follows that
\[ \mathrm{rank}\,(H_a(X) \otimes H_b (Y))=\mathrm{rank}\,H_a(X) \, \mathrm{rank}\,H_b (Y),\]
which implies that
\begin{equation}\label{prod}
p_{X\times Y}(t) = p_{X}(t)\,p_{Y}(t).
\end{equation}

\begin{proposition}\label{pr:homologies}
Let $X, Y$ be connected oriented closed manifolds of dimension $n\geq2$.
Then
\begin{align*}
    H_0(X\sharp Y)&=H_n(X\sharp Y)=\mathbb{Z},
    &
    H_i(X\sharp Y) &= H_i(X) \oplus H_i(Y),\ (1\leq i\leq n-1).
\end{align*}
In particular, if $X,Y\in\mathcal{Z}$, then $X\sharp Y\in\mathcal{Z}$ and
\begin{equation}\label{connected_sum}
p_{X \sharp Y}(t) = 1 + \sum_{i=1}^{n-1} \bigl(\beta_i(X) + \beta_i(Y)\bigr) t^i + t^n.
\end{equation}
\end{proposition}
\begin{proof}
The identities $H_0(X\sharp Y)=H_n(X\sharp Y)=\mathbb{Z}$ follow from the fact that $X\sharp Y$ is a connected orientable manifold.

Let us prove the remained statements.
Let $D_1^n \subset X$, $D_2^n\subset Y$ be closed $n$-disks,
\[ \pi: (X\setminus \mathrm{int}\,D_1) \cup (Y\setminus \mathrm{int}\,D_2)\to X\sharp Y \]
be the natural quotient map, and
\begin{align*}
    X' &= \pi(X\setminus \mathrm{int}\,D_1), &
    Y' &= \pi(Y\setminus \mathrm{int}\,D_2), &
    \mathbb{S}^{n-1}&=X'\cap Y' \subset X \sharp Y.
\end{align*}

We will show that $H_{n}(X') = 0$ and that for any $1\leq i\leq n-1$ the group $H_i(X\sharp Y)$ is isomorphic to  $H_i(X') \oplus H_i(Y')$.
Then, due to the fact that $X$ is homeomorphic with $X \sharp \mathbb{S}^{n}$, the group $H_i(X)$ will be isomorphic with $H_i(X')\oplus H_i(D_1)$ which in turn is isomorphic with $H_i(X')$.
This will prove that $H_i(X\sharp Y)$ is isomorphic with $H_i(X) \oplus H_i(Y)$ for $1\leq i\leq n-1$, that will immediately imply~\eqref{connected_sum}.

Consider the exact Mayer-Vietoris sequence for $(X\sharp Y; X', Y')$ (see, for example,~\cite[(17.7)]{GeenbergHarper}:
\[
\cdots \to
H_i(\mathbb{S}^{n-1}) \xrightarrow[z \mapsto (\alpha_i(z), \beta_i(z))]{\mathbf{A}_i}
H_i(X')\oplus H_i(Y')
\xrightarrow{\mathbf{B}_i}
H_i(X\sharp Y) \xrightarrow{\partial_i}
H_{i-1}(\mathbb{S}^{n-1})  \xrightarrow{\mathbf{A}_{i-1}} \cdots
\]
where $\alpha_i$ and $\beta_i$ are induced by embeddings $\mathbb{S}^{n-1} \subset X'$ and $\mathbb{S}^{n-1} \subset Y'$, respectively.
Then, due to the fact that $H_i(\mathbb{S}^{n-1})=0$ for $i\not=0, n-1$, we obtain that $\mathbf{B}_{n-1}$ is an epimorphism, $\mathbf{B}_i$ is an isomorphism for  $2\leq i \leq n-2$, and $\mathbf{B}_{1}$  is a monomorphism.

On the other hand, since $\mathbb{S}^{n-1} = \partial X' = \partial Y'$, it follows that $\alpha_{n-1}$, $\beta_{n-1}$, and, therefore, $\mathbf{A}_{n-1}$, are zero homomorphisms.
Hence $\mathbf{B}_{n-1}$ is injective and therefore is an isomorphism.
Moreover, as $\mathbb{S}^{n-1}$, $X'$, and $Y'$ are connected, it follows that $\alpha_0$ and $\beta_0$ are isomorphisms.
Therefore $\mathbf{A}_{0}$ is injective.
Hence, $\mathbf{B}_{1}$ is surjective and therefore is an isomorphism.

Finally, since $\mathbf{A}_{n-1} = 0$, we see that
\[
    \partial_{n}:\mathbb{Z}=H_{n}(X\sharp Y) \to H_{n-1}(\mathbb{S}^{n-1}) = \mathbb{Z}
\]
is a self-epimorphism of $\mathbb{Z}$, whence it must be an isomorphism.
Hence
\[
    \mathbf{A}_{n}: 0 = H_n(\mathbb{S}^{n-1}) \to H_{n}(X')\oplus H_{n}(Y')
\]
is a surjection of a zero group, and thus $H_{n}(X')=H_{n}(Y')=0$.
\end{proof}

\begin{corollary}\label{invind}
$H_i(\mathcal{S}^n_g)$, $n\geq3$, is trivial for $2\leq i \leq n-2$, isomorphic to $\mathbb{Z}^g$ for $i\in \{1,n-1\}$, and to $\mathbb{Z}$ for $i\in \{0,n\}$.
\end{corollary}
\begin{proof}
Recall that $\mathbb{S}^{k} \in \mathcal{Z}$ for $k\geq1$ and $p_{\mathbb{S}^k}(t) = 1 + t^k$.
Then applying~\eqref{prod} we get that
\[
    p_{\mathbb{S}^{n-1}\times\mathbb{S}^{1}}(t) =
    p_{\mathbb{S}^{n-1}}(t)\,p_{\mathbb{S}^{1}}(t) = (1+t^{n-1})(1+t) = 1 + t + t^{n-1} + t^n.
\]
Hence, due to~\eqref{connected_sum},
\[
    p_{\mathcal{S}^n_{g}}(t) = 1 + g t + g t^{n-1} + t^n
\]
for $g\geq0$.
The latter formula means that $H_0(\mathcal{S}^n_g)=H_{n}(\mathcal{S}^n_g)=\mathbb{Z}$, $H_1(\mathcal{S}^n_g)=H_{n-1}(\mathcal{S}^n_g)=\mathbb{Z}^g$ and $H_i(\mathcal{S}^n_g)=0$ for $2\leq i \leq n-2$.
\end{proof}

\section{Intersection number}
Let $X,Y$ be orientable smooth closed submanifolds of an orientable smooth manifold $M^n$ such that the intersection $X\cap Y$ is transversal and $\dim\,X + \dim\, Y=n$.
Then $X\cap Y$ consists of finitely many points.
To each intersection point $p\in X\cap Y$ we associate a number $j_p$ equal to $1$ ($-1$), whenever the orientation of the tangent space $T_x M^n$ agrees with (resp. is opposite to) the orientation induced by orientation of direct sum $T_p X \oplus T_p Y$.
Then the following sum
\[
    \sum\limits_{p\in X\cap Y} j_p
\]
is called the \emph{intersection number} of submanifolds $X$ and $Y$.

The following statement is a consequence of~\cite[Theorem~I, \S 70]{ZeTr}.

\begin{proposition}\label{ind_eq}
Let $S^p, S^q$ be spheres of dimensions $p, q$ smoothly embedded into an orientable smooth manifold $M^n$, such that $p+q=n$, and $S^q$ is homological to zero.
Then the intersection number of spheres $S^p, S^q$ equals to zero.
\end{proposition}

\section{Proof of Theorem~\ref{good_index}}
Let $f^t$ be a gradient-like flow on $\mathcal{S}^n_{g}$, $g>0$, having no heteroclinic intersections, that is invariant manifolds of distinct saddle equilibria of $f^t$ do not intersect each other.
We should show that the set of its saddle equilibria consists of points whose Morse index equals either $1$ or $(n-1)$.

Suppose the contrary, that is the non-wandering set of $f^t$ contains a saddle $\sigma$ with Morse index $i\in \{2,\dots,n-2\}$.
Since the invariant manifolds of distinct saddle equilibria do not intersect each other, it follows from~\cite[Theorem~2.1]{Sm} that there exists a unique pair $\alpha, \omega$ of sink and source equilibrium such that $cl\,{W^{s}_{\sigma}}=W^s_{\sigma}\cup \alpha$ and $cl\,{W^{u}_{\sigma}}=W^u_{\sigma}\cup \omega$.
Therefore, the closures of the stable and unstable manifolds of $\sigma$ are spheres of dimensions $n-i$ and $i$, respectively, and those spheres are smoothly embedded at all points, except possibly $\alpha$ and $\omega$.
Notice that those spheres intersect transversally only at the point $\sigma$, hence their intersection number equals either $1$ or $-1$.
On the other hand, due to Corollary~\ref{invind}, $H_{i}(\mathcal{S}^n_g) = H_{n-i}(\mathcal{S}^n_g)=0$ for $2\leq i \leq n-2$.
Hence, due to Proposition~\ref{ind_eq}, the intersection number of spheres $cl\,{W^{u}_{\sigma}}$, $cl\,{W^{s}_{\sigma}}$ must be zero.
This contradiction proves the absence of any saddle equilibrium states with Morse index equal to $i\in \{2,\dots,n-2\}$.

Let $c_i$ be a number of equilibria of $f^t$ having Morse index equals $i$.
Then $\nu_{f^t} = c_1+c_{n-1}$ and $\mu_{f^t} = c_0 +c_{n}$.
Moreover, due to Morse inequalities (see~\cite[Theorem~4.1]{Smale60a}), $c_i\geq\beta_i(\mathcal{S}^{n}_{g})$.
According to Corollary~\ref{invind},  $c_1,c_{n-1}\geq g$, so $\nu_{f^t} = 2g+k$ for some $k\geq0$.
Hence, $\mu_{f^t} = \nu_{f^t} +2 -2g = k+2$.


\begin{thebibliography}{1}

\bibitem{BGMP02}
C.~Bonatti, V.~Grines, V.~Medvedev, and E.~P\'{e}cou, \emph{Three-manifolds
  admitting {M}orse-{S}male diffeomorphisms without heteroclinic curves},
  Topology Appl. \textbf{117} (2002), no.~3, 335--344. \MR{1874094}

\bibitem{GeenbergHarper}
Marvin~J. Greenberg and John~R. Harper, \emph{Algebraic topology}, Mathematics
  Lecture Note Series, vol.~58, Benjamin/Cummings Publishing Co., Inc.,
  Advanced Book Program, Reading, Mass., 1981, A first course. \MR{643101}

\bibitem{GrGuPoMa-non}
V.~Grines, E.~Gurevich, O.~Pochinka, and D.~Malyshev, \emph{On topological
  classification of {M}orse-{S}male diffeomorphisms on the sphere {$S^n$} {$(n
  > 3)$}}, Nonlinearity \textbf{33} (2020), no.~12, 7088--7113. \MR{4173571}

\bibitem{GrGuPo-matan}
V.~Z. Grines, E.~A. Gurevich, and O.~V. Pochinka, \emph{Topological
  classification of {M}orse-{S}male diffeomorphisms without heteroclinic
  intersections}, J. Math. Sci. (N.Y.) \textbf{208} (2015), no.~1, Problems in
  mathematical analysis. No. 79 (Russian), 81--90. \MR{3392132}

\bibitem{Pil78}
S.~Ju. Piljugin, \emph{Phase diagrams that determine {M}orse-{S}male systems
  without periodic trajectories on spheres}, Differencialnye Uravnenija
  \textbf{14} (1978), no.~2, 245--254, 386. \MR{0501181}

\bibitem{ZeTr}
Herbert Seifert and William Threlfall, \emph{Seifert and {T}hrelfall: a
  textbook of topology}, Pure and Applied Mathematics, vol.~89, Academic Press,
  Inc. [Harcourt Brace Jovanovich, Publishers], New York-London, 1980,
  Translated from the German edition of 1934 by Michael A. Goldman, With a
  preface by Joan S. Birman, With ``Topology of $3$-dimensional fibered
  spaces'' by Seifert, Translated from the German by Wolfgang Heil. \MR{575168}

\bibitem{Sm}
S.~Smale, \emph{Differentiable dynamical systems}, Bull. Amer. Math. Soc.
  \textbf{73} (1967), 747--817. \MR{228014}

\bibitem{Smale60a}
Stephen Smale, \emph{Morse inequalities for a dynamical system}, Bull. Amer.
  Math. Soc. \textbf{66} (1960), 43--49. \MR{117745}

\bibitem{Sm61}
\bysame, \emph{On gradient dynamical systems}, Ann. of Math. (2) \textbf{74}
  (1961), 199--206. \MR{133139}

\end{thebibliography}

\providecommand{\bysame}{\leavevmode\hbox to3em{\hrulefill}\thinspace}
\providecommand{\MR}{\relax\ifhmode\unskip\space\fi MR }
\providecommand{\MRhref}[2]{%
  \href{http://www.ams.org/mathscinet-getitem?mr=#1}{#2}
}
\providecommand{\href}[2]{#2}

\end{document}